\documentclass[a4paper]{article}
\usepackage{amsmath,amsfonts,amsthm, amssymb}
\usepackage{color}
\newtheorem{theorem}{Theorem}[section] 
\newtheorem{lemma}[theorem]{Lemma} 
\newtheorem{proposition}[theorem]{Proposition}

\newenvironment{proofof}[1]{\normalsize {\it Proof of #1}:}{{\hfill $\Box$}}
\newcommand{\R}{{\mathbb R}}

\newcommand{\N}{{\mathbb N}}
\newcommand{\C}{{\mathbb C}}
\newcommand{\CG}{{\mathbb C}G}
\def\Factors#1#2{{\cal F}_{#1}^{#2}}
\def\JFactors{{\cal F}}
\def\FF#1#2{{\rm F}_{#1}^{#2}}
\newcommand{\MF}{{\rm MF}}

\newenvironment{mylist}{\begin{list}{}{
\setlength{\parskip}{0mm}
\setlength{\topsep}{1mm}
\setlength{\parsep}{0mm}
\setlength{\itemsep}{0.5mm}
\setlength{\labelwidth}{7mm}
\setlength{\labelsep}{3mm}
\setlength{\itemindent}{0mm}
\setlength{\leftmargin}{12mm}
\setlength{\listparindent}{6mm}
}}{\end{list}}

\def \K {{\cal K}}
\title{Rapid Decay is Preserved by Graph Products}
\author{Laura Ciobanu, Derek F. Holt and Sarah Rees}
\date{31st October 2011}
\begin{document}
\maketitle
\begin{abstract}
We prove that the rapid decay property (RD) of groups is preserved by graph
products defined on finite simplicial graphs.
\medskip



\end{abstract}

\section{Introduction}
\label{intro}

The main result proved in this paper is that the graph product of
finitely many groups with the rapid decay property (RD) has the
same property. This is stated formally in Section~\ref{sec:mainresult}
as Theorem~\ref{RD}. 

We call a function $\ell:G \rightarrow \R$ a {\it length function} for
a group $G$ if it satisfies
\[ \ell(1_G)=0,\quad\ell(g^{-1})=\ell(g),\quad
\ell(gh) \leq \ell(g)+ \ell(h), \quad \forall g,h \in G. \]
Following Jolissaint \cite{J90}, a group $G$ is said to have RD if
the operator norm $||.||_*$ for the group algebra $\CG$ is bounded by a
constant multiple
of the Sobolev norm $||.||_{2,r,\ell}$, a norm that is a variant of the
$l^2$ norm
weighted by a length function $\ell$ for $G$.

More precisely, RD holds for $G$ if there are constants $C,r$ and a
length function $\ell$ on $G$ such that
for any $\phi,\psi \in \C{G}$,
\[ ||\phi||_* := \sup_{\psi \in \CG}\frac{|| \phi*\psi||_2}{||\psi||_2}
\leq C||\phi||_{2,r,\ell}.\]
Here,
$\phi*\psi$ denotes the convolution of $\phi$ and $\psi$,
$||.||_2$ the
standard $l^2$ norm, and $ ||.||_{2,r,\ell}$ the Sobolev norm of order
$r$ with respect to $\ell$. So we have:
\begin{eqnarray*}
\phi*\psi(g)&=&\sum_{h \in G}\phi(h)\psi(h^{-1}g),\\
||\psi||_2 &=& \sqrt{\sum_{g \in G}|\psi(g)|^2},\\
||\phi||_{2,r,\ell}&=& \sqrt{ \sum_{g \in G} |\phi(g)|^2 (1 +
\ell(g))^{2r}}.
\end{eqnarray*}

A brief introduction to RD is given in \cite{CR05};
it is relevant to the Baum-Connes and Novikov conjectures (see an
account in \cite{V02}).
RD was originally studied by Jolissaint \cite{J90},
after it emerged from work of Haagerup \cite{Haagerup}, who proved
it for free groups. Jolissaint extended Haagerup's methods to
prove it for classical hyperbolic groups (i.e. discrete
cocompact groups of isometries of hyperbolic space), and to prove
that free and direct products of groups with RD inherit
the property, as do subgroups; de la Harpe \cite{H88} extended Jolissaint's arguments for classical hyperbolic groups to
derive RD for word hyperbolic groups. More recently, Drutu
and Sapir have shown that a group that is hyberbolic relative to a family of
parabolic subgroups with RD must itself have RD \cite{DS05}.
In all this work the proofs focus on the
factorisations of geodesic words. Other authors prove RD through
examination of actions of the group \cite{RamRobSteg, CR05}.

The graph product construction is a natural generalisation of both direct
and free products. Given a finite simplicial graph $\Gamma$ with a group
attached to each vertex, the associated \textit{graph product} is the
group generated by the vertex groups with the added relations that
elements of groups attached to adjacent vertices commute;
the representation of such a group as a graph product of directly
indecomposable groups is proved to be unique \cite{Radcliffe}.

Right-angled Artin groups \cite{Charney07}
(also known as graph groups) and right-angled Coxeter groups arise in this
way,
as the graph products of infinite cyclic groups and cyclic groups of order 2
respectively, and have been widely studied; some groups with
rather interesting properties arise via graph products,
including a group (a subgroup of a right-angled Artin group)
that has FP but is not finitely presented \cite{BestvinaBrady} and a
group ($F_2 \times F_2$) with insoluble subgroup membership
problem \cite{Mihailova}.
Both right-angled Artin groups
and right-angled Coxeter groups are already known
to possess RD, through their actions on CAT(0) cube
complexes \cite{CR05} (indeed all finite rank Coxeter groups possess RD for
this reason \cite{Chatterji}).

Graph products were introduced by Green in her PhD thesis \cite{G90} where,
in particular, a normal form was developed and the graph product construction
was shown to preserve residual finiteness; this work was extended by Hsu and
Wise in \cite{HsuWise} where, in particular, right-angled Artin groups
were shown
to embed in right-angled Coxeter groups and hence to be linear.
The preservation of semihyperbolicity, automaticity (as well as
asynchronous automaticity and biautomaticity) and the possession of a
complete rewrite system under graph products is proved in \cite{HM95},
necessary and sufficient conditions for the preservation of
hyperbolicity in \cite{Meier}, the
question of when the group is virtually free in
\cite{LohreySenizergues}, of orderability in \cite{Chiswell}.
Automorphisms and the structure of centralisers for graph products
of groups have been the subject of recent study, and in particular
graph products defined over random graphs have provoked some
interest for their applications
\cite{CostaFarber,CharneyRuaneStambaughVijayan,CharneyFarber}.

Our proof that the graph product construction preserves rapid decay will
build on the methods used by Jolissaint
for direct and free products \cite{J90}.
We use a reformulation of RD due to Jolissaint, 
explained in Section~\ref{sec:reformulation} below, which compares
norms on elements of $\CG$ with restricted support.
We examine in Section~\ref{graphprod} the geodesic decompositions of
elements in a graph product, and
show that they satisfy particularly useful properties, which will be
applied in our proof.
Section~\ref{sec:mainresult}
states our main result, Theorem~\ref{RD}, and reduces its proof to the 
proof of a further technical condition, which is similar to that of
Jolissaint's reformulation, but easier to verify in the context of graph
and free products; the same condition is used in \cite{J90} for free products.
The remaining two sections are devoted to the proof of
Proposition~\ref{lambda_condition},
thereby completing the proof of Theorem~\ref{RD} 

\section{A reformulation of the rapid decay property}
\label{sec:reformulation}

By~\cite[Lemma 2.1.3]{J90}, any length function $\ell$ on $G$ is equivalent
to one with $\ell(G) \subseteq \N$ and $\ell(g) >0\ \forall g \ne 1$.
Since RD is invariant under length equivalence, by \cite[Remark 1.1.7]{J90},
we shall assume from now on that all length functions have this property.

Given a length function $\ell$ on $G$,
and $k\in \N$, we define $C_k(\ell)$ to be the set $\{g\in G \mid
\ell(g)=k \}$.
We write $\chi_k$ for the characteristic function on $C_k$,
and for $\phi \in \CG$, we write $\phi_k$ for the pointwise product
$\phi.\chi_k$.

It is proved by Jolissaint \cite[Proposition 1.2.6]{J90} that RD for $G$
is equivalent to the following condition:

$(*)\quad$ There exist $c,r>0$  such that
   $\forall \phi,\psi \in \CG,\ k,l,m \in \N:$
\[
\begin{array}{ccll}
||(\phi_k * \psi_l)_m ||_2 &\leq& c||\phi_k||_{2,r,\ell} || \psi_l||_2&
{\rm if}\ |k-l|\leq m \leq k + l,\\
||(\phi_k * \psi_l)_m ||_2  &= &0&{\rm otherwise.}
\end{array} \]

It follows from the properties of a length function
that $||(\phi_k * \psi_l)_m ||_2 =0$ for $m$ outside
the range $[|k-l|,k+l]$. Hence we shall establish RD by verifying
the following condition:

$(**)\quad$ There exists a polynomial $P(x)$
such that $ \forall \phi,\psi \in \CG,\ k,l,m \in \N:$
\[ |k-l|\leq m \leq k + l \quad \Rightarrow \quad
||(\phi_k * \psi_l)_m ||_2 \leq P(k)||\phi_k||_2 || \psi_l||_2.  \]

\section{Graph products}
\label{graphprod}

Graph products of groups are studied in detail in \cite{G90}, and we
shall make use of the results from that thesis.
Let $\Gamma= (V,E)$ be a finite simplicial graph, together with vertex groups
$G_v$ for each $v \in V$. The associated graph product $G$ is defined to be the
quotient of the free product of the groups $G_v$ by the normal closure of
all the commutators $[g,g']$ for which $g \in G_{v}$, $g' \in G_w$ and
$\{v,w\}$ is an edge of the graph. We write $G=G(\Gamma; G_v, v \in
V)$.

Given such a group $G$, we define $\K$ to be the set of
cliques (including the empty set)
of the associated graph $\Gamma$, which we can identify with a set
of subsets of the vertex set $V$ of $\Gamma$.
We define $\K_m$ to be the subset of $\K$ of cliques of size $m$.
Given any subset $J$ of $V$ we define $G_J$ to be the subgroup
of $G$ generated by the elements of its subgroups $G_v$ with $v \in J$.
It follows from~\cite[Proposition 3.31]{G90} that $G_J$ is naturally isomorphic
to the graph product defined by the induced subgraph of $\Gamma$ with
vertex set $J$.  If $J$ is a clique, 
then $G_J$ is the direct product of its vertex groups.

Every element of the group $G$ can be written as a product
$y_1\cdots y_k$ for $k\geq 0$ and with each $y_i$ in a vertex group $G_{v_i}$;
that is, each element of the group has a representation
as a word 
over the set $S= \cup_{v \in V} (G_v
\setminus \{ 1 \})$.
We call such a representation an {\em expression}, the elements
$y_i$ the {\em syllables} of the expression, and say that the expression has
{\em syllable length} $k$; we define the syllable length of $g$,
$\lambda(g)$,
to be the minimum of the syllable lengths of expressions for $g$,
and say that an expression for $g$ is {\em reduced} if it has
syllable length $\lambda(g)$.
The function $\lambda$ is easily seen to be a length function, but it is
not the one that we shall use to prove RD.

We define $\Lambda_k = \{ g \in G: \lambda(g)=k\}$.

Note that when $G$ is a free product, every expression for which
consecutive $y_i$s come from distinct vertex groups is reduced,
and each reduced expression corresponds to a distinct element of $G$.
That is not true in general. But
it is proved in \cite[Theorem 3.9]{G90} that any expression for $g$
can be transformed to any reduced expression by a sequence of replacements
of the form $y'y\rightarrow yy'$ where $y,y'$ belong to commuting vertex
groups, or $y'y''\rightarrow y$, where $y'y''=y$ is a relation
holding between three elements of one of the vertex groups, or deletion
of $y'y''$ where the $y',y''$ are mutually inverse elements of a vertex group.
Hence any reduced expression for an element $g$ involves the same
syllables, but the order of the syllables in the expression is not
determined.

We shall need to estimate the number of factorisations of elements in graph products.
For $g \in \Lambda_{k+l}$, we use the notation
\begin{eqnarray*}
\Factors{k}{l}(g)&:=&\{(g_1,g_2) \mid g=_G g_1 g_2,\,
\lambda(g_1)=k,\,\lambda(g_2)=l\},\\
\FF{k}{l}&:=&\sup_{g \in \Lambda_{k+l}} |\Factors{k}{l}(g)|.
\end{eqnarray*}
Similarly, given a clique $J \in \K$ and  $g \in \Lambda_{k+l+|J|}$,
we use the notation
\begin{eqnarray*}
\Factors{k}{l}(J,g)&:=&\{(g_1,s,g_2) \mid g=_G g_1s g_2,\,
\lambda(g_1)=k,\,\lambda(g_2)=l,\, s \in G_J,\, \lambda(s)=|J|\},\\
\FF{k}{l}(J)&:=&\sup_{g \in \Lambda_{k+l+|J|}} |\Factors{k}{l}(J,g)|.
\end{eqnarray*}

By considering factorisations of $g^{-1}$, we see that $\FF{k}{l}$ = $\FF{l}{k}$
and $\FF{k}{l}(J)$ = $\FF{l}{k}(J)$.

Let $g,h \in G$. We say that $h$ is a left divisor of $g$, 
if $\lambda(g) = \lambda(h) + \lambda(h^{-1}g)$, or,
equivalently, if $h$ has a reduced expression $v$ that is a prefix
of a reduced expression $w$ for $g$.
We define right divisors similarly.

\begin{lemma}
\label{P1}
For each $J \in \K$, 
$\FF{k}{l}(J)$ is bounded by a polynomial in $\min(k,l)$.
\end{lemma}
\begin{proof}
Since $\FF{k}{l}(J) = \FF{l}{k}(J)$, it is sufficient to prove that
$\FF{k}{l}(J)$ is bounded by a polynomial in $k$.

Let $g \in \Lambda_m$ with $m\geq k+l$, and
consider factorisations $g=g_1sg_2$ with $g_1 \in \Lambda_k$,
$g_2 \in \Lambda_l$, and $s \in G_J$ with $\lambda(s)=|J|$.

We start by bounding the number of left divisors $g_1$ of $g$ with
$g_1 \in \Lambda_k$. Suppose that $y_1 y_2 \cdots y_m$ is a reduced
expression for $g$, with each $y_i$ an element of a vertex group. Then
$g_1=y_{\sigma(1)} \cdots y_{\sigma(k)},$ where $\sigma$ is a permutation of
$\{1, \ldots, m\}.$
When we transform the original expression $y_1 y_2 \cdots y_m$ to the new
expression, using shuffles, there is no need to swap any of the syllables
in $g_1$ among themselves, so we can assume that
$\sigma(1) \leq \sigma(2) \leq \cdots \leq \sigma(k)$.

Let us call the syllable $y_{\sigma(i)}$ with $1 \le i \le k$
{\em unconstrained} if its vertex group commutes with the vertex groups of
each of $y_{\sigma(i+1)},\ldots,y_{\sigma(k)}.$
(In particular, $y_{\sigma(k)}$ is unconstrained.) 

We claim that $g_1$ is uniquely determined by its unconstrained syllables
$y_{\sigma(i)}$ with $1 \le i \le k$.
To show this, suppose that $g_1,g_1' \in \Lambda_k$ are left divisors of $g$
corresponding to permutations $\sigma$, $\sigma'$ of
$\{1, \ldots, m\},$ where $g_1,g_1'$ have the same unconstrained syllables.
Suppose that $g_1 \ne g'_1$, and let $i$ be maximal such that $y_{\sigma(i)}$
is a syllable of $g_1$ but not of $g'_1$.  Then, by assumption, $y_{\sigma(i)}$
is not unconstrained, so there exists $j$ with $i < j \le k$ such that
$y_{\sigma(i)}$ does not commute with $y_{\sigma(j)}$. By maximality of $i$,
$y_{\sigma(j)}$ is a syllable of $g'_1$. Since $y_{\sigma(i)}$ and
$y_{\sigma(j)}$ do not commute,
$y_{\sigma(i)}$ must remain to the left of $y_{\sigma(j)}$ after any shuffles
of $g$. So $y_{\sigma(i)}$ must be a syllable of $g'_1$, a contradiction.

If there were two unconstrained syllables in $g_1$ from the same
vertex group $G_v$, they could be moved together, thereby shortening
the expression; hence there is at most one. Since $g_1$ has syllable length $k$,
an unconstrained syllable $y_{\sigma(i)}$ in
$g_1$ from $G_v$ must be among the first $k$
syllables in $y_1y_2 \cdots y_m$ that come from $G_v$.
Hence there are $k+1$ possible ways of selecting (0 or 1) syllables of 
$y_1y_2 \cdots y_m$ 
from $G_v$ that will be unconstrained in $g_1$.
This gives $(k+1)^{|V|}$ choices in total for the
unconstrained syllables of $g_1$, and hence $(k+1)^{|V|}$ choices for the left
divisor $g_1$.

Similarly, we have an upper bound of $(|J|+1)^{|V|}$ for the number of left
divisors of length $|J|$ of $g_1^{-1}g$ for a given $g_1$. In fact, the bound
for the number of choices for $s$ is $(|J|+1)^{|J|}$, as $s\in G_J$.
The right divisor $g_2$ is then completely determined by $g,g_1$ and $s$.
This results in the inequality
$\FF{k}{l}(J) \leq (k+1)^{|V|}(|J|+1)^{|J|},$
giving the required polynomial bound in $k$.
\end{proof}

\begin{lemma}\label{P2}
Suppose that $g, h_1, h_2 \in G$ with $g=_G h_1h_2$,
$h_1 \in \Lambda_k, h_2\in \Lambda_l$, and $g \in \Lambda_{k+l-q}$ with
$q\ge 0$.  Then $h_1=_G g_1s_1w$ and $h_2=_G w^{-1}s_2g_2$,
where:
\begin{mylist}
\item[(1)] $s_1,s_2 \in J$  for some $J \in \K$, 
and $\lambda(s_1) = \lambda(s_2) = \lambda(s_1s_2) = |J|$,
\item[(2)] $q = |J|+2\lambda(w)$.
\end{mylist}
\end{lemma}
\begin{proof}
Suppose first that there is no cancellation between $h_1$ and $h_2$, that is,
when we reduce the expression $h_1h_2$, no syllable becomes equal to the
identity. Then pick a right
divisor $s_1$ of $h_1$, and a left divisor $s_2$ of $h_2$,
both minimal with respect to syllable length,
such that $\lambda(s_1s_2) = \lambda(s_1) + \lambda(s_2) -q$.
(If $q=0$ then $s_1$ and $s_2$ are empty.)
Minimality ensures that
every syllable of $s_1$ (or $s_2$) must merge without cancelling with a symbol
of $s_2$  (or $s_1$). So $\lambda(s_1)=\lambda(s_2)=\lambda(s_1s_2)$,
the syllables within each of $s_1$ and $s_2$ must commute,
and there is a clique $J\subset \Gamma$ of size $q$ such that
$s_1, s_2 , s_1s_2 \in G_J$.

If there is cancellation between $h_1$ and $h_2$, then we can
find a right divisor $w$ of $h_1$ such that $w^{-1}$ is a left divisor of
$h_2$. We choose $w$ of maximal syllable length, so that
$h_1=h'_1w$ and $h_2=w^{-1}h'_2$ for some
$h'_1, h'_2 \in G$,  and no cancellation occurs between $h'_1$ and $h'_2$.
Applying the above argument to $h'_1$ and $h'_2$, we find a clique $J$ of size
$q-2\lambda(w)$,
a right divisor $s_1$ of $h'_1$, and a left divisor $s_2$ of $h'_2$,
with $\lambda(s_1)=\lambda(s_2)=\lambda(s_1s_2)=|J|$.
\end{proof}

\section{Main result}
\label{sec:mainresult}

We prove
\begin{theorem}
\label{RD}
Suppose that $G=G(\Gamma; G_v, v \in V)$ is a graph product of
groups defined with respect to a finite simplicial graph $\Gamma=(V,E)$,
and suppose that each vertex group $G_v$ satisfies RD.
Then $G$ satisfies RD.
\end{theorem}

The proof generalises Jolissaint's proof in \cite{J90} of RD for free products,
which is itself a generalisation of Haagerup's proof for free groups.
We use the result of \cite[Lemma 2.1.2]{J90}, which implies (although it
is not stated explicitly) that RD is preserved by taking
direct products, and rely on the existence of polynomial bounds relating to
factorisation in graph products proved in Lemma~\ref{P1} and Lemma~\ref{P2}.

The following easy consequence of the Cauchy-Schwarz inequality will
be used frequently.
\begin{lemma}\label{CS}
For any positive integer $M$ and real numbers $a_1, \dots, a_M$,
\[\left(\sum_{i=1}^M a_i\right)^2 \leq M \left(\sum_{i=1}^M a_i^2\right).\]
\end{lemma}

The proof of Theorem~\ref{RD} that follows depends on Proposition
\ref{lambda_condition}, which is stated within the proof; we defer the
proof of that to Section~\ref{sec:Prop}.

\begin{proofof}{Theorem~\ref{RD}}
For each $v \in V$ we choose a length function $\ell_v$ on $G_v$ with respect
to which $G_v$ has rapid decay. Then, given $g \in G$ and a reduced expression
$g=y_1\cdots y_k$ for $G$, with $y_j \in G_{v_j}$, we define
\[ \ell(g) = \sum_{j=1}^k \ell_{v_j}(y_j). \]
That $\ell$ is both well defined and a length function follows easily
from \cite[Theorem 3.9]{G90}.
On subgroups $G_J$ with $J \in \K$,
we see that $\ell$ restricts to $\ell_J$, defined in the same way as a sum of
functions $\ell_{v_j}$ with $v_j \in J$; it follows from
~\cite[Lemma 2.1.2]{J90} that
$G_J$ satisfies RD with respect to $\ell_J$. 

We prove rapid decay with respect to the length
function $\ell$ by verifying the condition
$(**)$ of Section~\ref{sec:reformulation}.
But rather than prove that directly we shall deduce it from a
condition on the length function $\lambda$, that is stated as 
Proposition~\ref{lambda_condition}.

Recall that $\Lambda_k$ is defined as the set $\{g \in G: \lambda(g)=k\}$.
We now define
$\chi_{(k)}$ to be the characteristic function on $\Lambda_k$ and
$\phi_{(k)}$
to be the pointwise product $\phi.\chi_{(k)}$.
In general a function labelled with the subscript $(k)$ is understood
to have support on $\Lambda_k$.

We shall prove the following in Section~\ref{sec:Prop}.
\begin{proposition}
\label{lambda_condition}
$\exists\ c,r>0 $ such that
for all $\phi,\psi \in \CG,k,l,m \in \N$, and 
$|k-l|\leq m \leq k + l$,
\[ 
||(\phi_{(k)} * \psi_{(l)})_{(m)}||_2 \leq c||\phi_{(k)}||_{2,r,\ell} ||
\psi_{(l)}||_2. \]
\end{proposition}

To derive our main result from this proposition, we shall
use it to deduce the condition $(**)$ of Section~\ref{sec:reformulation},
that is, we shall deduce the similar condition in which restrictions to
$\Lambda_k,\Lambda_l,\Lambda_m$ are replaced by restrictions to $C_k,C_l,C_m$.
(But note that both of the length functions $\lambda$ and $\ell$ are
involved in the proposition statement.)
Jolissaint's proof of RD for free products follows exactly the
same strategy, and the argument below is basically his, with some slight
modification of notation to match our own.
We suppose that $k,l,m$ are fixed in the appropriate range.
We write $\phi'$ rather than $\phi_k$ and $\psi'$ rather than $\psi_l$,
to make our notation less cumbersome.

Since $\ell(g) \ge 1$ for all $g \in G \setminus\{1\}$, we have
$C_k \subseteq \cup_{j=0}^k \Lambda_j$ and $C_l \subseteq
\cup_{i=0}^l \Lambda_i$, and hence $\phi' = \sum_{j=0}^k \phi'_{(j)}$,
$\psi'=\sum_{i=0}^l \psi'_{(i)}$.
Similarly, for fixed $j$, we have
$|(\phi'_{(j)} * \psi')_m(g)| \le | \sum_{p=0}^m (\phi'_{j} * \psi')_{(p)}(g)|$,
for all $g \in G$.
(Note that we dropped the restriction to $C_m$ on the right hand side of
this inequality.)
Hence
\[ ||(\phi'_{(j)} * \psi')_{m}||_2^2 \leq
\sum_{p=0}^m || (\phi'_{(j)}*\psi')_{(p)}||_2^2.\]

Now, in any product of group elements $g_1g_2 = g$ with $g_1 \in \Lambda_i$,
$g_2 \in \Lambda_j$, $g \in \Lambda_p$, we must have
$p\leq i+j$, $i\leq p+j$, $j \leq p+i$.
Hence in each of the terms $(\phi'_{(j)}*\psi')_{(p)}$ 
in the sum on the right hand side of the
above inequality, for fixed $j$ and $p$, the support of $\psi'$ lies in
the union of the $\Lambda_i$ with $|j-p| \le i \le j+p$,
and so
\[ \sum_{p=0}^m || (\phi'_{(j)}*\psi')_{(p)}||_2^2
\leq 
 \sum_{p=0}^m || \sum_{i=|j-p|}^{j+p} (\phi'_{(j)}*\psi'_{(i)})_{(p)}||_2^2\]
Since there are at most $2j+1$ values
of $i$ in each of the ranges $[|j-p|,j+p]$,
we can use Lemma ~\ref{CS} to bound the right hand side above by
\[ (2j+1) \sum_{p=0}^m \sum_{i=|j-p|}^{j+p}
                 || (\phi'_{(j)}*\psi'_{(i)})_{(p)}||_2^2.\]

It follows from Proposition~\ref{lambda_condition} that, for
$|j-i| \leq p \leq j+i$,
\[ ||(\phi'_{(j)} * \psi'_{(i)})_{(p)}||_2 \leq c||\phi'_{(j)}||_{2,r,\ell} ||
\psi'_{(i)}||_2.\] For other values of $p$, the left hand side is zero.  
Hence
\[
||(\phi'_{(j)} * \psi')_{m}||_2^2 
\leq c^2(2j+1) ||\phi'_{(j)}||^2_{2,r,\ell} \sum_{p=0}^m
\sum_{i=|j-p|}^{j+p} || \psi'_{(i)}||_2^2.\]
Since $|i-j| \le p \le i+j$, for a given value of $i$, there are at most
$2j+1$ values of $p$ in the above summation, and so we have
\[
||(\phi'_{(j)} * \psi')_{m}||_2^2 \leq
    c^2(2j+1)^2 ||\phi'_{(j)}||^2_{2,r,\ell} 
           \sum_{i=0}^{j+m} || \psi'_{(i)}||_2^2 \leq
    c^2(2j+1)^2 ||\phi'_{(j)}||^2_{2,r,\ell} 
           || \psi'||_2^2.\]
Now, using the triangle inequality together with Lemma~\ref{CS} again,
we have
\begin{eqnarray*}
||(\phi_k * \psi_l)_{m}||_2^2 =
||(\phi' * \psi')_{m}||_2^2 &\leq& (\sum_{j=0}^k
                             ||(\phi'_{(j)} * \psi')_m||)^2\\
&\leq& (k+1) \sum_{j=0}^k ||(\phi'_{(j)} * \psi')_{m}||^2 \\
&\leq& c^2(k+1)\sum_{j=0}^k(2j+1)^2 ||\phi'_{(j)}||^2_{2,r,\ell}
           || \psi'||_2^2 \\
&\leq& c^2(k+1)(2k+1)^2 ||\phi_k||^2_{2,r,\ell}
           || \psi_l||_2^2 \\
&=& P(k) ||\phi_k||^2_2
           || \psi_l||_2^2,
\end{eqnarray*}
where the polynomial $P$ has degree $3+2r$.
So we have deduced $(**)$.

The proof of the theorem will be complete once
Proposition~\ref{lambda_condition} is proved.
\end{proofof}

\section{Technicalities of the proof of Proposition \ref{lambda_condition}}
In an attempt to make the proof of Proposition \ref{lambda_condition}
more readable we start with some technical results and definitions.

For a function $\phi_{(k)}$ 
and $p \ge 0$, we can define functions $\phi^{(p)}_{(k-p)}$ and
${}^{(p)}\phi_{(k-p)}$  by 

\begin{eqnarray*}
\phi^{(p)}_{(k-p)}(u) &=&
\left\{
\begin{array}{rl}
\sqrt{\sum_{w \in \Lambda_{p}} | \phi_{(k)}(uw)|^2} & \quad {\rm if} \
   u \in \Lambda_{k-p}\\
0 &\quad {\rm otherwise} \\
\end{array}
\right.
\\
{}^{(p)}\phi_{(k-p)}(u) &=&\left\{ \begin{array}{rl} \sqrt{\sum_{w \in
\Lambda_{p}} | \phi_{(k)}(w^{-1}u)|^2} & \quad\textrm{if}\ 
u \in \Lambda_{k-p}\\
0 & \quad\hbox{\rm otherwise}\end{array} \right.
\end{eqnarray*}

Note that these functions are non-negative real-valued; we shall sometimes
make use of that fact as we bound sums.

\begin{lemma}
\label{relate_l2_norms}
\[||\phi^{(p)}_{(k-p)}||_2^2\ \leq\ \FF{k-p}{p}||\phi_{(k)}||_2^2 \quad{\rm and}
\quad ||{}^{(p)}\phi_{(k-p)}||_2^2 \ \leq\ \FF{k-p}{p}||\phi_{(k)}||_2^2. \]
\end{lemma}
\begin{proof}
\begin{eqnarray*}
||\phi^{(p)}_{(k-p)}||_2^2 &=& \sum_{u \in \Lambda_{k-p}}\sum_{w \in
\Lambda_{p}}|\phi_{(k)}(uw)|^2 \\ & \leq & \FF{k-p}{p} \sum_{h \in
\Lambda_{k}} |\phi_{(k)}(h)|^2\\
&=& \FF{k-p}{p} ||\phi_{(k)}||_2^2\\
\end{eqnarray*}
The second inequality follows similarly, since $\FF{k-p}{p} = \FF{p}{k-p}$.
\end{proof}

For $0 \le i \le k$ and  $g \in \Lambda_{k-i}$, we define functions
$\phi^g_{(i)}$ and ${}^g\phi_{(i)}$  by 
\begin{eqnarray*}
\phi^{g}_{(i)}(v) &=& \left\{ \begin{array}{rl}
    \phi_{(k)}(vg) & \quad {\rm if}\ vg \in \Lambda_{k}\\
0 &  \quad\hbox{\rm otherwise}\end{array} \right.
\\
{}^{g}\phi_{(i)}(v) &=& \left\{ \begin{array}{rl}
    \phi_{(k)}(gv) & \quad {\rm if}\ gv \in \Lambda_{k}\\
0 & \quad\hbox{\rm otherwise}\end{array} \right.
\end{eqnarray*}

\section{Proof of Proposition \ref{lambda_condition}}
\label{sec:Prop}
\begin{proof}

Suppose that $m=k+l-q$, with $q \ge 0$, and that
$g \in \Lambda_{k+l-q}$.
By Lemma~\ref{P2}, for each factorisation of $g$ as
a product $h_1h_2$ with $h_1 \in \Lambda_k$, $h_2 \in \Lambda_l$,
there is a 5-tuple $(g_1,g_2,s_1,s_2,w)$ of elements with syllable
lengths $k-q+p,\,l-q+p,\,q-2p,\,q-2p,\,p$,
for which $h_1= g_1s_1w$, $h_2=w^{-1}s_2g_2$, $s:= s_1s_2 \in G_J$ has
syllable length $q-2p$, and $s_1,s_2,s$ are all elements of $G_J$ for some
$J \in \K_{q-2p}$.

For ease of notation we now
define, for $s \in G_J$,
\[
\JFactors(J,s) :=\{(s_1,s_2) \in G_J \times G_J:  s= s_1s_2,
   \lambda(s_1)=\lambda(s_2)=|J|\}.
\]

Now for any $g \in \Lambda_{k+l-q}$,
$|\phi_{(k)} * \psi_{(l)}(g)|$
is bounded above by
\[
\sum^{p=\lfloor{q/2}\rfloor}_{\scriptsize\begin{array}{c}p=1,\\
J \in \K_{q-2p}\end{array}}
\sum_{\scriptsize\begin{array}{c}(g_1,s,g_2)\in\\
\Factors{k-q+p}{l-q+p}(J,g)\end{array}}
\sum_{(s_1,s_2) \in \JFactors(J,s)}\,
\sum_{w \in \Lambda_p}
|\phi_{(k)} (g_1s_1w)\psi_{(l)} (w^{-1}s_2g_2)|,
\]
by the triangle inequality.

By Cauchy-Schwarz
\begin{eqnarray*}
\sum_{w \in \Lambda_p}
|\phi_{(k)} (g_1s_1w)\psi_{(l)} (w^{-1}s_2g_2)|&\leq&
\sqrt{ \sum_{w \in \Lambda_p} |\phi_{(k)} (g_1s_1w)|^2}
\sqrt{ \sum_{w \in \Lambda_p} |\psi_{(l)} (w^{-1}s_2g_2)|^2}\\
&=&
\phi^{(p)}_{(k-p)}(g_1s_1) \times {}^{(p)}\psi_{(l-p)}(s_2g_2)\\
&=&
{}^{g_1}\phi^{(p)}_{(q-2p)}(s_1) \times {}^{(p)}\psi^{g_2}_{(q-2p)}(s_2).
\end{eqnarray*}
Further,
\begin{eqnarray*}
\sum_{(s_1,s_2) \in \JFactors(J,s)}\,
{}^{g_1}\phi^{(p)}_{(q-2p)}(s_1)\times {}^{(p)}\psi^{g_2}_{(q-2p)}(s_2)
&\leq&
{}^{g_1}\phi^{(p)}_{(q-2p)}* {}^{(p)}\psi_{(q-2p)}^{g_2}(s),\\
\end{eqnarray*}
where the convolution here is over $G_J$, not over $G$.

Then we apply the Lemma~\ref{CS} to see that 

\begin{eqnarray*}
|\phi_{(k)} * \psi_{(l)}(g)|^2
&\leq&
\left (
\sum^{p=\lfloor{q/2}\rfloor}_{\scriptsize\begin{array}{c}p=1,\\
J \in \K_{q-2p}\end{array}}
\sum_{\scriptsize\begin{array}{c}(g_1,s,g_2)\in\\
\Factors{k-q+p}{l-q+p}(J,g)\end{array}}
{}^{g_1}\phi^{(p)}_{(q-2p)}* {}^{(p)}\psi_{(q-2p)}^{g_2}(s)\right)^2\\
&\leq& \MF(k,q,l)
\sum^{p=\lfloor{q/2}\rfloor}_{\scriptsize\begin{array}{c}p=1,\\
J \in \K_{q-2p}\end{array}}
\sum_{\scriptsize\begin{array}{c}(g_1,s,g_2)\in\\
\Factors{k-q+p}{l-q+p}(J,g)\end{array}}
\left({}^{g_1}\phi^{(p)}_{(q-2p)}* {}^{(p)}\psi_{(q-2p)}^{g_2}(s)\right)^2.\\
\end{eqnarray*}
where $\MF(k,q,l):=
\sum_{p=1}^{\lfloor q/2\rfloor} \sum_{J \in \K_{q-2p}} \FF{k-q+p}{l-q+p}(J)$.
Since there are only finitely many cliques $J$, it follows from
Lemma~\ref{P1} that $\MF(k,q,l)$ is bounded by $Q(k)$ for some polynomial $Q$.
Hence

\begin{eqnarray*}
||(\phi_{(k)} * \psi_{(l)})_{(m)}||_2^2
&=&\sum_{g \in \Lambda_m}|\phi_{(k)} * \psi_{(l)}(g)|^2 \\
&\leq & Q(k)
\sum^{p=\lfloor{q/2}\rfloor}_{\scriptsize\begin{array}{c}p=1,\\
J \in \K_{q-2p}\end{array}}
\sum_{\scriptsize\begin{array}{c}g_1 \in \Lambda_{k-q+p}\\s \in
G_J,\\
g_2 \in \Lambda_{l-q+p}\end{array}}
({}^{g_1}\phi^{(p)}_{(q-2p)}* {}^{(p)}\psi_{(q-2p)}^{g_2}(s))^2\\
&=& Q(k)
\sum^{p=\lfloor{q/2}\rfloor}_{\scriptsize\begin{array}{c}p=1,\\
J \in \K_{q-2p}\end{array}}
\sum_{\scriptsize\begin{array}{c}g_1 \in \Lambda_{k-q+p}\\
g_2 \in \Lambda_{l-q+p}\end{array}}
||{}^{g_1}\phi^{(p)}_{(q-2p)}* {}^{(p)}\psi_{(q-2p)}^{g_2})||_{2;G_J}^2.\\
\end{eqnarray*}
But now, since RD holds with respect to $\ell_J$ in each of the groups $G_J$,
we have

\begin{eqnarray*}
|| {}^{g_1}\phi^{(p)}_{(q-2p)}* {}^{(p)}\psi^{g_2}_{(q-2p)} ||_{2;G_J}^2
&\leq&
c_J^2|| {}^{g_1}\phi^{(p)}_{(q-2p)}||^2_{2,r_J,\ell_J;G_J}\quad||
{}^{(p)}\psi^{g_2}_{(q-2p)} ||_{2;G_J}^2.\\
\end{eqnarray*}
We deduce easily from this that
\begin{eqnarray*}
\sum_{\scriptsize\begin{array}{c}g_1 \in \Lambda_{k-q+p},\\g_2\in
\Lambda_{l-q+p}\end{array}}
|| {}^{g_1}\phi_{(q-2p)}* \psi^{g_2}_{(q-2p)} ||_{2;G_J}^2
&\leq& c_J^2
\sum_{g_1 \in \Lambda_{k-q+p}} ||
{}^{g_1}\phi^{(p)}_{(q-2p)}||^2_{2,r_J,\ell_J;G_J} \times\\
&&\hspace*{15pt}
\sum_{g_2\in \Lambda_{l-q+p}} ||{}^{(p)}\psi^{g_2}_{(q-2p)} ||_{2;G_J}^2.
\end{eqnarray*}
Then, for $c=\max c_J$ and $r = \max r_J$, we have
\begin{eqnarray*}
&&\sum^{p=\lfloor{q/2}\rfloor}_{\scriptsize\begin{array}{c}p=1,\\
J \in \K_{q-2p}\end{array}}
\sum_{\scriptsize\begin{array}{c}g_1 \in \Lambda_{k-q+p}\\
g_2 \in \Lambda_{l-q+p}\end{array}}
|| {}^{g_1}\phi^{(p)}_{(q-2p)}* {}^{(p)}\psi^{g_2}_{(q-2p)} ||_{2;G_J}^2\\
&\leq&
c^2 \sum^{p=\lfloor{q/2}\rfloor}_{\scriptsize\begin{array}{c}p=1,\\
J \in \K_{q-2p}\end{array}}
\left(
\sum_{g_1 \in \Lambda_{k-q+p}} ||
{}^{g_1}\phi^{(p)}_{(q-2p)}||^2_{2,r_J,\ell_J;G_J}
\sum_{g_2\in \Lambda_{l-q+p}}||{}^{(p)}\psi^{g_2}_{(q-2p)}
||_{2;G_J}^2\right) \\
&\leq& c^2
\left (
\sum^{p=\lfloor{q/2}\rfloor}_{\scriptsize\begin{array}{c}p=1,\\
J \in \K_{q-2p}\end{array}}
\sum_{g_1 \in \Lambda_{k-q+p}} ||
{}^{g_1}\phi^{(p)}_{(q-2p)}||^2_{2,r,\ell_J;G_J}\right) \times\\
&&\hspace*{100pt}
\left (
\sum^{p=\lfloor{q/2}\rfloor}_{\scriptsize\begin{array}{c}p=1,\\
J \in \K_{q-2p}\end{array}}
\sum_{g_2\in \Lambda_{l-q+p}}|| {}^{(p)}\psi^{g_2}_{(q-2p)} ||_{2;G_J}^2
\right).
\end{eqnarray*}
But
\begin{eqnarray*}
&&\sum^{p=\lfloor{q/2}\rfloor}_{\scriptsize\begin{array}{c}p=1,\\
J \in \K_{q-2p}\end{array}}
\sum_{g_1 \in \Lambda_{k-q+p}} ||
{}^{g_1}\phi^{(p)}_{(q-2p)}||^2_{2,r,\ell_J;G_J}\\
&=& \sum^{p=\lfloor{q/2}\rfloor}_{\scriptsize\begin{array}{c}p=1,\\
J \in \K_{q-2p}\end{array}}
\sum_{\scriptsize\begin{array}{c}g_1 \in \Lambda_{k-q+p},\\s \in G_J\\
w \in \Lambda_p \end{array}}
| \phi_{(k)}(g_1sw)|^2(1+\ell_J(s))^{2r}.
\end{eqnarray*}
Since $\lfloor{q/2}\rfloor \le k$ and the number of sets $J$ is bounded,
Lemma~\ref{P1} implies that the number
of factorisations of a fixed $g' \in \Lambda_k$ as $g_1sw$ in the
above sum is at most $P(k)$ for some polynomial $P$.
So the sum is bounded above by
\[
P(k) \sum_{g' \in \Lambda_{k}} | \phi_{(k)}(g')|^2(1+\ell(g'))^{2r}
= P(k)  ||\phi_{(k)}||^2_{2,r,\ell},
\]
and similarly
\[
\sum^{p=\lfloor{q/2}\rfloor}_{\scriptsize\begin{array}{c}p=1,\\
J \in \K_{q-2p}\end{array}}
\sum_{g_2\in \Lambda_{l-q+p}} ||{}^{(p)}\psi^{g_2}_{(q-2p)} ||_{2;G_J}^2
\leq P(k) ||\psi_{(l)}||^2_2.
\]
Hence
\begin{eqnarray*}
|| (\phi_{(k)}*\psi_{(l)})_{(m)}||_2^2
&\leq&
c^2Q(k)P(k)^2 ||\phi_{(k)}||^2_{2,r,\ell} ||\psi_{(l)}||^2_2\\
&\leq&
c^2||\phi_{(k)}||^2_{2,r+\deg(Q)+2\deg(P),\ell} ||\psi_{(l)}||^2_2,
\end{eqnarray*}
where the final inequality uses the fact that $k \le \ell(g)$ for all
$g \in \Lambda_k$.
This completes the proof of Proposition~\ref{lambda_condition}.
\end{proof}

\section*{Acknowledgments}

The first author was partially supported by the Marie Curie Reintegration
Grant 230889. The third author would like to thank the Mathematics
Departments
of the Universities of Fribourg and Neuch\^atel, Switzerland, for their
hospitality and generous support.

\bigskip

\textsc{L. Ciobanu,
Mathematics Department,
University of Fribourg,
Chemin du Muse\'e 23,
CH-1700 Fribourg, Switzerland
}

\emph{E-mail address}{:\;\;}\texttt{Laura.Ciobanu@unifr.ch}
\medskip

\bigskip

\textsc{D. F. Holt,
Mathematics Institute,
University of Warwick,
Coventry CV4 7AL,
UK
}

\emph{E-mail address}{:\;\;}\texttt{D.F.Holt@warwick.ac.uk}
\medskip
\bigskip

\textsc{S. Rees,
School of Mathematics and Statistics,
University of Newcastle,
Newcastle NE1 7RU,
UK
}

\emph{E-mail address}{:\;\;}\texttt{Sarah.Rees@newcastle.ac.uk}
\medskip

\end{document}